\documentclass[10pt,english]{extarticle}

\usepackage{caption}
\usepackage{enumerate}
\usepackage{t1enc}
\usepackage{babel}
\usepackage{amsthm}
\usepackage{amsmath}
\usepackage{amsfonts}
\usepackage{mathrsfs}
\usepackage[cp1250]{inputenc}
\usepackage{amsthm}
\usepackage{amsmath}
\usepackage{amssymb}
\usepackage{amsfonts}
\usepackage{indentfirst}
\usepackage{textcomp}
\usepackage{mathabx}
\usepackage{bbm}
\usepackage{tikz}
\usepackage{xcolor}
\usepackage{extsizes}
\usepackage[textwidth=18.5cm,textheight=23cm,left=3.5cm,voffset=-1cm,,footskip=1.5cm,headheight=0cm]{geometry}
\usepackage{fancyhdr}
\usepackage{yfonts}
\usepackage{float}
\usepackage{pgf}
\usepackage{pgfplots}
\usepackage{amsfonts,amsmath,amssymb,graphicx,array,url}
\usepackage{polski}
\usepackage[cp1250]{inputenc}
\pgfplotsset{compat=1.12}

\makeatletter
\def\namedlabel#1#2{\begingroup
    #2%
    \def\@currentlabel{#2}%
    \phantomsection\label{#1}\endgroup
}
\makeatother

\begin{document}

\title{Long run control of nonhomogeneous Markov processes}

\author{{\L}ukasz Stettner  \footnote{Institute of Mathematics
 Polish Academy of Sciences,
  Sniadeckich 8, 00-656 Warsaw, Email: stettner@impan.pl,  research supported by National Science Centre, 
  Poland by NCN grant
  2024/53/B/ST1/00703}
}
\maketitle

\begin{abstract}In the paper average reward per unit time and average risk sensitive reward functionals are considered for controlled nonhomogeneous Markov processes. Existence of solutions to suitable Bellman equations is shown. Continuity of the value functions with respect to risk parameter is also proved. Finally stability of functionals with respect to pointwise convergence of Markov controls is studied. 
\end{abstract}



\theoremstyle{plain}
\setlength{\parskip}{12pt plus0pt minus12pt}

\newtheorem{theorem}{Theorem}
\newtheorem{lemma}[theorem]{Lemma}
\newtheorem{corollary}[theorem]{Corollary}
\newtheorem{proposition}[theorem]{Proposition}
\newtheorem{remark}{Remark}
\newtheorem{example}{Example}
\def\epr{\rightline $\box$}
\def\cF{\mathcal{F}}
\def\cT{\mathcal{T}}
\def\bF{\mathbb{F}}
\def\bP{\mathbb{P}}
\def\bE{\mathbb{E}}
\def\bR{\mathbb{R}}
\def\bT{\mathbb{T}}
\def\bN{\mathbb{N}}
\def\bV{\mathbb{V}}
\def\ef{\mathcal{F}}
\def\ee{\mathbb{E}}
\def\er{\mathbb{R}}
\def\bse{\mathcal{E}}
\def\prob{\mathbb{P}}
\def\tao{\tau_{\cal{O}}}
\def\om{{\cal{O}}_m}
\def\o{\cal{O}}
\def\taom{\tau_{{\cal{O}}_m}}
\def\taomk{\tau_{{\cal{O}}_{m_k}}}
\def\ep{\epsilon}
\def\vep{\varepsilon}
\def\se{{\cal{E}}}

\maketitle

\section{Introduction}\label{S:introduction}
Let $(X_n)$ be a discrete time inhomogeneous Markov process on $(\Omega, F,(F_n),\bP)$ taking values in a complete separable metric space  $E$ endowed with Borel $\sigma$ field $\se$ and with controlled transition operator $\bP_n^{a_n}(X_n,\cdot)$ in generic time $n\in \mathbb{N}$ with compact set of control parameters $U$. Denote by $C(E)$ the space of continuous bounded functions on $E$.  We shall assume so called controlled Feller property i.e. for each $n\in \mathbb{N}$ and $f\in C(E)$ the mapping $E\times U\ni (x,a)\mapsto \int_E f(y)\bP_n^a(x,dy)$ is continuous.  Markov process is controlled using sequence $V=(a_0,a_1,\ldots, a_n\ldots)$ such that $a_n\in U$ is an $F_n$ adapted random variable. We denote by $\ee_{x}^V$ expected value corresponding to controlled Markov process with the use of the sequence $V$ and starting at time $0$ from $x$. We will be interested to maximize the following long run functionals: average reward per unit time
\begin{equation}\label{fun1}
J_x(V)=\liminf_{n\to \infty} {1\over n} \ee_x^V\left[\sum_{i=0}^{n-1}c_i(X_{i},a_{i})\right]
\end{equation}
or long run risk sensitive functional
\begin{equation} \label{fun2}
J_{x}^r(V,\gamma)=\liminf_{n\to \infty}{1 \over n \gamma}  \ln\left( \ee_x^V\left[\exp\left\{\gamma\sum_{i=0}^{n-1}c_i(X_{i},a_{i})\right\}\right]\right).
\end{equation}
where $c_i:E\times U\mapsto R$ is a continuous bounded function and $\gamma\neq 0$ is a risk factor. The functional \eqref{fun2} has important financial interpretation (see \cite{PitSte2022}) since it measures not only expected value but also higher moments with weights depending on the risk factor $\gamma$.  
To solve control problems with functionals \eqref{fun1} and \eqref{fun2} we have to study Bellman equations which are of the form for $n\in \mathbb{N}_0=\mathbb{N}\cup\left\{0\right\}$.
\begin{equation}\label{Bel1}
w_n(x)=\sup_{a\in U}\left[c_n(x,a)-\lambda_n +\int_E w_{n+1}(y)\bP_n^a(x,dy)\right]
\end{equation}
\begin{equation}\label{Bel2}
\tilde{w}_n(x,\gamma)=\sup_{a\in U} \left\{c_n(x,a)-\lambda_n^\gamma+  {1\over \gamma}\ln\left[ \int_E e^{\gamma \tilde{w}_{n+1}(y,\gamma)}\bP_n^a(x,dy)\right]\right\},
\end{equation}
where we are looking for functions $w_n$, $\tilde{w}_n(\cdot,\gamma)\in C(E)$ and constants $\lambda_n$, $\lambda_n^\gamma$ respectively. As one can see solutions  $w_n$, $w_n(\cdot,\gamma)$ to \eqref{Bel1} and \eqref{Bel2} are defined up to an additive constant. Therefore we shall look for such solutions in the space $C_{sp}(E)$ which is a quotient $C(E)$ with respect to the following equivalence relation: for $v_1,v_2\in C(E)$ we have $v_1\equiv v_2$ when $\|v_1-v_2\|_{sp}:=\sup_{x,x'\in E} v_1(x)-v_2(x)-(v_1(x')-v_2(x'))=0$. To study solutions so called additive and multiplicative Poisson equations we shall need the space $B(E)$ of bounded Borel measurable functions and its quotient space $B_{sp}(E)$ (with the same equivalence relation) endowed with the above defined span norm $\|\cdot\|_{sp}$.

We have the following verification result
\begin{theorem}\label{ver}
If there exist solutions $w_n\in C(E)$ and $\lambda_n$ to \eqref{Bel1} such that $\sup_n \|w_n\|_{sp}<\infty$ and $w_n(\bar{x})=0$ for $n=0,1,\ldots$ in a fixed $\bar{x}\in E$ then
\begin{equation}
\lambda:=\sup_V J_x(V)=\liminf_{n\to \infty} {1\over n} \sum_{i=0}^{n-1}\lambda_i
\end{equation}
and controls of the form
\begin{equation}\label{verst}
\hat{V}=(u_0(X_0),\ldots,u_n(X_n),\ldots),
\end{equation}
 where $u_n$ is a Borel measurable selector to the right hand side of the equation \eqref{Bel1}, are optimal.
Furthermore if there exist solutions $\tilde{w}_n(\cdot,\gamma)\in C(E)$ and $\lambda_n^\gamma$ to \eqref{Bel2} with $\gamma\neq 0$ such that $\sup_n \|\tilde{w}_n(\cdot, \gamma\|_{sp}<\infty$ and $\tilde{w}_n(\bar{x},\gamma)=0$ for $n=0,1,\ldots$ in a fixed $\bar{x}\in E$ then
\begin{equation}
\lambda^r(\gamma):=\sup_V J_x^r(V,\gamma)=\liminf_{n\to \infty} {1\over n} \sum_{i=0}^{n-1}\lambda_i^\gamma
\end{equation}
and  controls of the form 
\begin{equation}\label{verstr}
\hat{V}^\gamma=(u_0^\gamma(X_0),\ldots,u_n^\gamma(X_n),\ldots),
\end{equation}
 where $u_n^\gamma$ is a Borel measurable selector to the right hand side of the equation \eqref{Bel2}, are optimal.
\end{theorem}
\begin{proof} For any control $V=(a_0,a_1,\ldots, a_n\ldots)$ with $i=0,1,\ldots$ we have
\begin{equation}
w_i(X_i)\geq \left[c_i(X_i,a_i)-\lambda_i +\int_E w_{i+1}(y)\bP_i^{a_i}(X_i,dy)\right]
\end{equation}
with equality for $a_i=u_i(X_i)$. Then
\begin{eqnarray}\label{verbel}
&&\sum_{i=0}^{n-1} \lambda_i \geq \sum_{i=0}^{n-1} c_i(X_i,a_i)+ \sum_{i=0}^{n-2} \left(\int_E w_{i+1}(y)\bP_i^{a_i}(X_i,dy)- w_{i+1}(X_{i+1})\right)- \nonumber \\
&&w_0(X_0)+\int_E w_n(y)\bP_{n-1}^{a_{n-1}}(X_{n-1},dy)
\end{eqnarray}
and
\begin{equation}\label{verbel1}
\sum_{i=0}^{n-1} \lambda_i \geq \ee_x^V\left[\sum_{i=0}^{n-1} c_i(X_i,a_i)-w_0(X_0)+w_n(X_n)\right] 
\end{equation}
Since $\|w_n\|:=\sup_{x\in E}|w_n(x)|\leq \|w_n\|_{sp}$ we have that $\sup_n \|w_n\|<\infty$. Therefore  dividing both sides of \eqref{verbel1} by $n$ and letting $n\to \infty$ we obtain that 
$\lambda \geq \liminf_{n\to \infty} {1\over n} \sum_{i=0}^{n-1}\lambda_i$
with equality for \eqref{verst}. 

Consider now the functional \eqref{fun2}. For any strategy $V=(a_0,a_1,\ldots, a_n\ldots)$ we have 
\begin{equation}\label{Bel2}
\tilde{w}_i(X_i,\gamma)\geq c_i(X_i,a_i)-\lambda_i^\gamma+  {1\over \gamma}\ln\left[ \int_E e^{\gamma \tilde{w}_{i+1}(y,\gamma)}\bP_i^{a_i}(X_i,dy)\right],
\end{equation}
with equality for $a_i=u_i^\gamma(X_i)$.Then
\begin{eqnarray}
&& |\gamma| \sum_{i=0}^{n-1} c_i(X_i,a_i) \geq  |\gamma|\sum_{i=0}^{n-1}\lambda_i^\gamma - sgn(\gamma)\ln \left(\prod_{i=0}^{n-2}\left[\int_E e^{\gamma(\tilde{w}_{i+1}(y,\gamma)-\tilde{w}_{i+1}(X_{i+1},\gamma))}\bP_i^{a_i}(X_i,dy)\right] \right.  \nonumber \\
&&\left. \int_E e^{\gamma \tilde{w}_n(X_{n-1},\gamma)} \bP_{n-1}^{a_{n-1}}(X_{n-1},dy)\right)+|\gamma|\tilde{w}_0(X_0,\gamma)
\end{eqnarray}
where $sgn(\gamma)=1 $ for $\gamma>0$, $sgn(\gamma)=-1$ for $\gamma<0$ and $0$ otherwise.
Then
\begin{eqnarray}\label{verbel2}
&&{1\over n \gamma} \ln \left(\ee_x^V\left[e^{\gamma\sum_{i=0}^{n-1} c_i(X_i,a_i)}\right]\right)\leq  {1\over n} \sum_{i=0}^{n-1} \lambda_i^\gamma + {1\over n}sgn(\gamma) \tilde{w}_0(X_0,\gamma) - \nonumber \\
&&{1\over n \gamma}\ln \left( \ee_x^V \left[\prod_{i=0}^{n-2}\left[\int_E e^{\gamma(\tilde{w}_{i+1}(y,\gamma)-\tilde{w}_{i+1}(X_{i+1},\gamma))}\bP_i^{a_i}(X_i,dy)\right] e^{\gamma\tilde{w}_n(X_{n},\gamma)}\right]\right)\leq \nonumber \\
&& {1\over n} \sum_{i=0}^{n-1} \lambda_i^\gamma + {1\over n}sgn(\gamma) \tilde{w}_0(X_0,\gamma) +{1\over n} \|\tilde{w}_n\|   
\end{eqnarray}
Since $\sup_n \|\tilde{w}_n\|<\infty$ letting $n\to \infty$ we obtain that $J_x^\gamma(V)\leq \liminf_{n\to \infty} {1\over n} \sum_{i=0}^{n-1} \lambda_i^\gamma$, with equality for the strategy $\hat{V}^\gamma$ defined in \eqref{verstr}. 
This completes the proof.
\end{proof}

There is a number of papers concerning long run behaviour of nonhomogeneous Markov processes (see \cite{Dob}, \cite{Ios} or recently \cite{LiuLu}).  
There are not too many works concerning control of nonhomogeneous Markov processes (see \cite{Park}, \cite{Cao}, \cite{Zheng}, \cite{Ding}, \cite{Ngoc}, \cite{Bao}). The papers \cite{Ding}, \cite{Ngoc}, \cite{Bao} consider stability of such kind models with partial observation. Control of finite or countable nonhomogeneous Markov processes with average cost per unit time was studied in \cite{Park}, \cite{Cao} (and references therein). In the paper  \cite{Zheng} general zero sum game with average criterion was considered. This paper generalizes \cite{Zheng} in various directions. First of all additionally to average reward per unit time risk sensitive functionals are studied. Continuity of value functions and functions in the Bellman equations with respects to risk parameter is shown. Finally stability of functionals with respect to pointwise convergence of Markov controls is proved, which his crucial with respect to potential approximations. 

The paper is organized as follows. We first prove a technical theorem which allows us to obtains solutions to Bellman equations both for average reward per unit time as well as long run risk sensitive functional.   Then using this theorem we find solutions to average reward per unit time and long run risk sensitive Bellman equations. For risk sensitive control we use results of \cite{DiMasi  1999} nad \cite{Ste2023}. Next we consider continuity of functions in the Bellman equations with respect to risk parameter. In the final part of the paper we show continuity of studied functionals with respect to pointwise convergence of Markov controls,  which is an extension of \cite{Ste2024}. 
   
\section{Main technical theorem}

Denote by $\|c_n\|_{sp}=\sup_{x,x'\in E} \sup_{a,a'\in U} c(x,a)-c(x',a')$. We  have
\begin{theorem}\label{mtech}

Assume we are given a sequence of operators $T_n: C_{sp}(E)\mapsto C_{sp}(E)$ such that
\begin{enumerate}
\item[({i})] $\|T_nv_1-T_nv_2\|_{sp}\leq \Delta_n \|v_1-v_2\|_{sp}$ for $v_1,v_2\in C_{sp}(E)$,
\item[({ii})] $\|T_n0\|_{sp}\leq \|c_n\|_{sp}$,
\item[({iii})]  $\lim_{k\to \infty} \Delta_n\Delta_{n+1}\ldots\Delta_{n+k}\to 0$,
\item[({iv})] $R_n:=\|c_n\|_{sp} + \sum_{i=0}^\infty \Delta_n\ldots\Delta_{n+i}\|c_{n+i+1}\|_{sp} <\infty$,
\end{enumerate}
Then there is a sequence $w_n\in C_{sp}(E)$ such that
\begin{equation}
\|T_n w_{n+1} - w_n\|_{sp}=0 \quad \mbox{for $n=0,1,\ldots$}.
\end{equation}
Moreover $\|w_n\|_{sp}\leq R_n$ and
\begin{equation}\label{impine}
\|T_nT_{n+1}\ldots T_{n+k-1}0-w_n\|_{sp}\leq  \Delta_n\ldots \Delta_{n+k-1}\|w_{n+k}\|_{sp}.
\end{equation}
Furthermore when $\sup_n \|w_n\|_{sp}<\infty$ functions $w_n$ are unique up to an additive constant.
\end{theorem}
\begin{remark}
Sufficient condition for $\sup_n\|w_n\|_{sp}<\infty$ is that $\sup_n R_n<\infty$. When $c=\sup_n \|c_n\|_{sp}<
\infty$ and $\sup_n \Delta_n:=\Delta < 1$  then $R_n\leq c \sum_{i=0}^\infty \Delta^{i}={c\over 1-\Delta}$.
\end{remark}
\begin{proof}
For a fixed $n\in \mathbb{N}$ consider a sequence of functions $(w_n^i)$ such that $w_n^0\equiv 0$, $w_n^k(x):=T_nw_{n+1}^{k-1}(x)$ for $k=1,2, \ldots$. Then we have
\begin{equation}
w_n^k(x)=T_nw_{n+1}^{k-1}(x)=T_nT_{n+1}w_{n+2}^{k-2}(x)=\ldots=T_nT_{n+1}\ldots T_{n+k-1}w_{n+k}^{0}(x)
\end{equation}
and therefore using (i)
\begin{equation}
\|w_n^k-w_n^{k+r}|_{sp}= \|T_nT_{n+1}\ldots T_{n+k-1}0-T_nT_{n+1}\ldots T_{n+k-1}w_{n+k}^{r}\|_{sp}\leq  \Delta_n\Delta_{n+1}\ldots\Delta_{n+k-1}\|w_{n+k}^{r}\|_{sp}.
\end{equation}
By (ii) and (i) we have (taking into account that $w_n^0\equiv 0$)
\begin{equation}
\|w_n^k\|_{sp}\leq \|T_nw_{n+1}^{k-1} - T_nw_{n+1}^0\|_{sp}+\|T_nw_{n+1}^0\|_{sp}\leq\Delta_n \|w_{n+1}^{k-1}\|_{sp}+\|c_n\|_{sp}
\end{equation}
and iterating, using (iv) we obtain
\begin{eqnarray}
&& \|w_n^k\|_{sp}\leq \Delta_n \Delta_{n+1} \|w_{n+2}^{k-2}\|_{sp} + \Delta_n\|c_{n+1}\|_{sp}+ \|c_{n}\|_{sp}\nonumber \\ && \leq \ldots
\leq  \Delta_n \Delta_{n+1}\ldots \Delta_{n+k-1}\|w_{n+k}^0\|_{sp}+\sum_{i=0}^{k-2} \Delta_n \ldots \Delta_{n+i}\|c_{n+i+1}\|_{sp} + \|c_n\|_{sp}\leq R_n
\end{eqnarray}
Consequently using (iii) and (iv) we have that for each $n$ the sequence $(w_n^k)$ is bounded in the span norm and satisfies Cauchy condition. Therefore for each $n$ there is a function $w_n\in C_{sp}(E)$ such that $\|w_n^k-w_n\|_{sp}\to 0$ as $k\to \infty$.
Furthermore we have
\begin{eqnarray}
&&\|T_nw_{n+1} - w_{n}\|_{sp}\leq \|T_nw_{n+1} - T_nw_{n+1}^{k-1}\|_{sp}+ \|T_nw_{n+1}^{k-1}-w_n^k\|_{sp}+ \|w_n^k-w_n\|_{sp} \\
&& \leq \Delta_n \|w_{n+1}-w_{n+1}^{k-1}\|_{sp} + \|w_n^k-w_n\|_{sp}\to 0 \nonumber
\end{eqnarray}
as $k\to \infty$, so that $\|T_nw_{n+1} - w_{n}\|_{sp}=0$. Since $\|w_n^k\|_{sp}\leq R_n$ we also have that
$\|w_n\|_{sp}\leq R_n$. Note that using (i) we have
\begin{equation}
\|T_nT_{n+1}\ldots T_{n+k-1}0-T_nT_{n+1}\ldots T_{n+k-1}w_{n+k}\|_{sp}\leq \Delta_n\ldots \Delta_{n+k-1}\|w_{n+k}\|_{sp}.
\end{equation}
from which \eqref{impine} follows.

Assume additionally that $\sup_n \|w_n\|_{sp}<\infty$ and that there is another sequence $\tilde{w}_n\in C_{sp}(E)$ such that $\|\tilde{w}_n-w_n\|_{sp} \neq 0$ for some $n\in \mathbb{N}$ and $\|T_n\tilde{w}_n-\tilde{w}_n\|_{sp}=0$ and $\sup_n \|\tilde{w_n}\|_{sp}<\infty$. Then by (iii)
\begin{eqnarray}
&& \|w_n-\tilde{w}_n\|_{sp}=\|T_nw_{n+1}-T_n\tilde{w}_{n+1}\|_{sp}\leq\nonumber \\
&& \Delta_n \|w_{n+1}-\tilde{w}_{n+1}\|_{x
sp} = \Delta_n \|T_{n+1}w_{n+2}-T_{n+1}\tilde{w}_{n+2}\|_{sp} \nonumber \\
&&\leq\Delta_n \Delta_{n+1} \|w_{n+2}-\tilde{w}_{n+2}\|_{sp} \leq \ldots \leq \Delta_n \Delta_{n+1} \ldots  \Delta_{n+k} \|w_{n+k+1}-\tilde{w}_{n+k+1}\|_{sp}\to 0
\end{eqnarray}
as $k\to \infty$ and $\|\tilde{w}_n-w_n\|_{sp} = 0$.
This completes the proof.
\end{proof}
\begin{corollary}\label{exist}
Assume additionally that $T_n(v+d)=T_nv+d$ for $v\in C_{sp}(E)$ and constant $d\in \mathbb{R}$. Then for a given $\bar{x}\in E$ there are sequences $w_n\in C_{sp}(E)$ and constants $\lambda_n$ such that $w_n(\bar{x})=0$ and
\begin{equation}\label{poiss}
w_n(x) = -\lambda_n + T_nw_{n+1}(x) \quad \mbox{for $x\in E$}.
\end{equation}
When $\sup_n R_n <\infty$ both functions $w_n$ and constants $\lambda_n$ are unique.
\end{corollary}
\begin{proof} Assume that $\|\hat{w}_n-T_n\hat{w}_{n+1}\|_{sp}=0$ for $n\in \mathbb{N}_0$. let
$\hat{\lambda}_n=-(\hat{w}_n-T_n\hat{w}_{n+1})$. Define $w_n(x)=\hat{w}_n(x)-\hat{w}_n(\bar{x})$.
Then
\begin{equation}
T_nw_{n+1}(x)=T_n\hat{w}_{n+1}(x)-\hat{w}_{n+1}(\bar{x})=  \hat{w}_n(x)+\hat{\lambda}_n-\hat{w}_{n+1}(\bar{x}) =  w_n(x) + (\hat{\lambda}_n - \hat{w}_{n+1}(\bar{x})+\hat{w}_n(\bar{x}))
\end{equation}
and \eqref{poiss} is satisfied with $\lambda_n=\hat{\lambda}_n - \hat{w}_{n+1}(\bar{x})+\hat{w}_n(\bar{x})$.
Uniqueness of $w_n$ and $\lambda_n$ follows directly from Theorem \ref{mtech}.
\end{proof}

We formulate now an analog of Theorem \ref{mtech} and Corollary \ref{exist} (with almost the same proofs) for the case when the operator $T_n$ transforms the space $B_{sp}(E)$ into itself. We have
\begin{proposition}\label{mtechp}
Assume there is a sequence of operators $T_n:B_{sp}(E)\mapsto B_{sp}(E)$ for which assumptions (i)-(iv) of Theorem \ref{mtech} and additionally property $T_n(v+d)=T_nv+d$ for $v\in B_{sp}(E)$ and constant $d\in \mathbb{R}$ are satisfied. Then for a given $\bar{x}\in E$ there are sequences $w_n\in B_{sp}(E)$ and constants $\lambda_n$ such that $w_n(\bar{x})=0$ and
\begin{equation}\label{poissn}
w_n(x) = -\lambda_n + T_nw_{n+1}(x) \quad \mbox{for $x\in E$}.
\end{equation}
Furthermore $\|w_n\|_{sp}\leq R_n$ and
\begin{equation}\label{impinen}
\|T_nT_{n+1}\ldots T_{n+k-1}0-w_n\|_{sp}\leq  \Delta_n\ldots \Delta_{n+k-1}\|w_{n+k}\|_{sp}.
\end{equation}
When $\sup_n R_n <\infty$ both functions $w_n$ and constants $\lambda_n$ are unique.
\end{proposition}

\section{Applications}

Let for $v\in C(E)$
\begin{equation}\label{opT}
T_nv(x) := \sup_{a\in U} \left[c(x,a)+\int_E v(y)\bP_n^a(x,dy)\right]
\end{equation}
We have
\begin{lemma}
For each $n\in \mathbb{N}$ the operator $T_n$ transforms $C(E)$ into $C(E)$ and for $v_1,v_2\in C(E)$ we have
\begin{equation}
\|T_n v_1 - T_nv_2\|_{sp}\leq \Delta_n \|v_1-v_2\|_{sp},
\end{equation}
with
\begin{equation}\label{defD}
\Delta_n:=\sup_{x,x'\in E} \sup_{a,a'\in U} \sup_{B\in \se} \bP_n^a(x,B)-\bP_n^{a'}(x',B).
\end{equation}
\end{lemma}
\begin{proof}
For $v_1,v_2\in C(E)$ and $x_1,x_2\in E$ we have 
\begin{eqnarray}
&&Tv_1(x_1)-Tv_2(x_1)-(Tv_1(x_2)-Tv_2(x_2))\leq \sup_{a\in U}\int_E (v_1(y)-v_2(y)) \bP_n^a(x_1,dy) - \nonumber \\
&&\inf_{a\in U} \int_E (v_1(y)-v_2(y)) \bP_n^a(x_2,dy)= \sup_{a,a'\in U} \int_E (v_1(y)-v_2(y)) (\bP_n^a(x_1,dy) - \bP_n^{a'}(x_2,dy))\leq  \nonumber \\
&&\sup_{a,a'\in U} \left[\sup_y (v_1(y)-v_2(y))(\bP_n^a(x_1,D(x_1,x_2,a,a')) - \right. \nonumber \\ &&\bP_n^{a'}(x_2,D(x_1,x_2,a,a')))-\inf_y (v_1(y)-v_2(y)) \nonumber \\
&&\left. (\bP_n^a(x_1,D^c(x_1,x_2,a,a')) - \bP_n^{a'}(x_2,D^c(x_1,x_2,a,a')))\right]\leq \|v_1-v_2\|_{sp} \Delta_n  
\end{eqnarray}
where the set $D(x_1,x_2,a,a')$ is from Hahn decomposition theorem of the measure $\bP_n^a(x_1,dy) - \bP_n^{a'}(x_2,dy)$. 
\end{proof}
\begin{theorem}\label{averp}
Assume that $\Delta_n$ defined in \eqref{defD} satisfies (iii) of Theorem \ref{mtech} and $\sup_n R_n<\infty$. Then there exist unique solutions: functions  $w_n\in C(E)$ such that for a fixed $\bar{x}\in E$ we have  $w_n(\bar{x})=0$ and constants $\lambda_n$  to the equation \eqref{Bel1}.
\end{theorem}
\begin{proof}
Notice that operator $T_n$ satisfies properties (i) and (ii) of Theorem \ref{mtech} and also $T_n(v+d)=T_nv+d$ for $v\in C(E)$ and constant $d\in \mathbb{R}$. Therefore we can use Corollary \ref{exist} which completes the proof.
\end{proof}
For $v(\cdot,\gamma)\in C(E)$ with fixed $\gamma\neq 0$ define the operator
\begin{equation}\label{Optr}
\tilde{T}_n v(x,\gamma):= \sup_{a\in U} \left[c_n(x,a)+{1\over \gamma}\ln\int_E e^{\gamma v(y,\gamma)}\bP_n^a(x,dy)\right]
\end{equation}
We have
\begin{lemma}\label{loccon}
For each $n\in \mathbb{N}$ and fixed $\gamma\neq 0$ the operator $\tilde{T}_n$ transforms $C(E)$ into $C(E)$ and for $v_1,v_2\in C(E)$ we have
\begin{equation}\label{est1r}
\|\tilde{T}_n v_1(\cdot,\gamma) - \tilde{T}_nv_2(\cdot,\gamma)\|_{sp}\leq \delta_n(|\gamma|(\|v_1\|_{sp}\vee \|v_2\|_{sp})) \|v_1(\cdot,\gamma)-v_2(\cdot,\gamma)\|_{sp},
\end{equation}
with
\begin{equation}\label{defDr}
\delta_n(|\gamma|(\|v_1\|_{sp}\vee \|v_2\|_{sp})):=\sup_{x,x'\in E} \sup_{a,a'\in U} \sup_{B\in \se} \nu_{x,a,\gamma v_1}^n(B)-\nu_{x',a',\gamma v_2}^n(B),
\end{equation}
where
\begin{equation}\label{defDr1}
\nu_{x,a,\gamma v_1}^n(B)={\int_B e^{\gamma v_1(y,\gamma)}\bP_n^a(x,dy) \over \int_E e^{\gamma v_1(y,\gamma)}\bP_n^a(x,dy)}.
\end{equation}
\end{lemma}
\begin{proof}
We follow the proof of Theorem 5 of \cite{Ste2023}. We need variational formula (1.15) of \cite{Dupuis}  for the operator $\tilde{T}_n$ separately in the case of $\gamma>0 $ and $\gamma<0$. In both cases we obtain \eqref{est1r} with $\delta_n$ defined in \eqref{defDr} and measures $\nu^n$ defined in \eqref{defDr1}.
\end{proof}

To show that operator $\tilde{T}_n$ defined in \eqref{Optr} satisfies properties required in Theorem \ref{mtech} we shall need additional assumption for $n\in \mathbb{N}$
\begin{equation}\label{eqn}
\sup_{x,x'\in E} \sup_{a\in U} \sup_{B\in \se} {\bP_n^a(x,B) \over \bP_n^a(x',B)}:=K_n<\infty
\end{equation}
Using the proof of Proposition 6 of \cite{Ste2023} we obtain
\begin{lemma}\label{defDrr}
Under \eqref{eqn} for $v(\cdot,\gamma)\in C_{sp}(E)$ and $\gamma\neq 0$  we have
\begin{equation}\label{estt}
\|\tilde{T}_nv(\cdot,\gamma)\|_{sp}\leq \|c_n\|_{sp} + {1\over |\gamma|} \ln K_n
\end{equation}
and consequently for fixed $\gamma\neq 0$ the operator $\tilde{T}_n$ is contraction in $C_{sp}(E)$ with constant
$\Delta_n^\gamma:=\delta(|\gamma|\|c_n\|_{sp} + \ln K_n)$
\end{lemma}
\begin{proof}
For $x_1,x_2\in E$ and $v(\cdot,\gamma)\in C(E)$ we have
\begin{eqnarray}
&&\tilde{T}_nv(x_1,\gamma)-\tilde{T}_nv(x_2,\gamma)\leq \sup_{a\in U}\left[ c_n(x_1,a)-c_n(x_2,a)+ \right. \nonumber \\
&&\left. {1\over \gamma} ln\left({\int_E v(y,\gamma) \bP_n^a(x_1,dy)\over \int_E v(y,\gamma) \bP_n^a(x_2,dy)}\right)\right]\leq  \|c_n\|_{sp} + {1\over |\gamma|} \ln K_n
\end{eqnarray}
which completes the proof of \eqref{estt}. The second part follows directly from Lemma \ref{loccon}.  
\end{proof}

\begin{theorem}\label{averpr}
Assume that $\Delta_n^\gamma$ defined in Lemma \ref{defDrr} satisfies (iii) of Theorem \ref{mtech} and $\sup_n R_n<\infty$. Then for fixed $\gamma\neq 0$  there exist unique: functions  $\tilde{w}_n(\cdot,\gamma)\in C(E)$ such that for a fixed $\bar{x}\in E$ we have  $\tilde{w}_n(\bar{x},\gamma)=0$ and constants $\lambda_n^\gamma$ which are solutions to the equation \eqref{Bel2}.
\end{theorem}

\begin{proof} By Lemma \ref{defDrr} operator $\tilde{T}_n$ with fixed $\gamma\neq 0$ satisfies (i) of Theorem 
\ref{mtech}. Also (ii) is satisfied and for constant $d\in \mathbb{R}$ we have  $\tilde{T}_n (v(x,\gamma)+d)=
\tilde{T}_n v(x,\gamma)+d$. Consequently by Corollary \ref{exist} we have required solutions to the Bellman equation \eqref{Bel2}.
\end{proof}
 
\section{Continuity with respect to $\gamma$}
In this section we study dependence of optimal values of the functionals \eqref{fun1}, \eqref{fun2} and solutions to the corresponding Bellman equations \eqref{Bel1}, \eqref{Bel2} on the risk parameter $\gamma$.   We assume that all assumptions of Theorems \ref{averp}   and \ref{averpr} are satisfied. Let  $\tilde{w}_n(x,0)=w(x)$ and $\tilde{w}(\bar{x},\gamma)=0$ for some $\bar{x}\in E$ and $n\in \mathbb{N}_0$ and $\lambda_n^0=\lambda_n$. We have
\begin{theorem} The mapping $E\times \mathbb{R}\ni (x,\gamma) \mapsto \tilde{w}_n(x,\gamma)$ is continuous and consequently  also the mapping $\mathbb{R} \ni \gamma \mapsto \lambda_n^\gamma$ is continuous.
\end{theorem}
\begin{proof}
Assume first that $x_m\to x$ and $\gamma_m\to \gamma\neq 0$ as $m\to \infty$.
From \eqref{impine} taking into account that $\|w_n\|_{sp}\leq \sup_i R_i$ we have for $n\in \mathbb{N}_0$, $k\in \mathbb{N}$ and any $\gamma' \in \mathbb{R}$
\begin{equation}\label{impesti}
\|\tilde{T}_n\tilde{T}_{n+1}\ldots \tilde{T}_{n+k-1}0(\cdot,\gamma')-\tilde{w}_n(\cdot,\gamma')\|_{sp}\leq  \Delta_n\ldots \Delta_{n+k-1}\sup_i R_i.
\end{equation}
By controlled Feller property and probabilistic interpretation we obtain
\begin{equation}
\tilde{T}_n\tilde{T}_{n+1}\ldots \tilde{T}_{n+k-1}0(x_m,\gamma_m)\to \tilde{T}_n\tilde{T}_{n+1}\ldots \tilde{T}_{n+k-1}0(x,\gamma)
\end{equation}
as $m\to \infty$. Therefore from \eqref{impesti} we have
\begin{equation}
\|\tilde{w}_n(\cdot,\gamma_m) - \tilde{w}_n(\cdot,\gamma)\|_{sp} \to 0
\end{equation}
as $m\to \infty$ and since $\tilde{w}_n(\bar{x},\gamma_m)=\tilde{w}_n(\bar{x},\gamma)=0$ we also obtain that
\begin{equation}
\|\tilde{w}_n(\cdot,\gamma_m) - \tilde{w}_n(\cdot,\gamma)\|\to 0
\end{equation}
as $m\to \infty$. Convergence $\lambda_n^{\gamma_m}\to \lambda_n$ follows directly from the controlled Feller property and Bellman equation \eqref{Bel2}. Notice that limits are uniquely defined and therefore do not depend on a particular sequence $\gamma_m\to \gamma$.
Assume now that $\gamma=0$. Analogously to \eqref{impesti} we have
 \begin{equation}\label{impesti2}
\|{T}_n{T}_{n+1}\ldots {T}_{n+k-1}0(\cdot)-\tilde{w}_n(\cdot,0)\|_{sp}\leq \Delta_n\ldots \Delta_{n+k-1}\sup_i R_i.
\end{equation}
Using Heffding's lemma (see Lemma 2.6 of \cite{Ma} or Lemma 2 of \cite{Ste2024}) we obtain for any control $V$ that
\begin{equation}\label{impesti3}
 0\leq \ln \left(\ee^V_x \left[\exp\left\{\gamma\sum_{i=n}^{n+k-1}c_i(X_{i-n},a_{i-n})\right\}\right]\right) -  \gamma \ee^V_x\left[\sum_{i=n}^{n+k-1}c_i(X_{i-n},a_{i-n})\right]\leq \left(\sum_{i=n}^{n+k-1}\|c_i\|_{sp}\right)^2 {\gamma^2 \over 8}
\end{equation}
Therefore
\begin{eqnarray}
&&\|\tilde{w}_n(\cdot,\gamma_m)-\tilde{w}_n(\cdot,0)\|_{sp}\leq\|\tilde{w}_n(\cdot,\gamma_m)-\tilde{T}_n\tilde{T}_{n+1}\ldots \tilde{T}_{n+k-1}0(\cdot,\gamma_m)\|_{sp}+ \nonumber \\
&&\|\tilde{T}_n\tilde{T}_{n+1}\ldots \tilde{T}_{n+k-1}0(\cdot,\gamma_m)-{T}_n{T}_{n+1}\ldots {T}_{n+k-1}0(\cdot)\|_{sp} +\|{T}_n{T}_{n+1}\ldots {T}_{n+k-1}0(\cdot)-\tilde{w}_n(\cdot,0)\|_{sp}\leq \nonumber \\
&& 2\Delta_n\ldots \Delta_{n+k-1}\sup_i R_i +  \left(\sum_{i=n}^{n+k-1}\|c_i\|_{sp}\right)^2 {|\gamma_m| \over 8}
\end{eqnarray}
where we  used \eqref{impesti},\eqref{impesti2} and \eqref{impesti3}. Letting $m\to \infty$ and then $k\to \infty$ by we obtain that
\begin{equation}
\|\tilde{w}_n(\cdot,\gamma_m)-\tilde{w}_n(\cdot,0)\|_{sp}\to 0
\end{equation}
as $m\to \infty$. Since $\tilde{w}_n(\bar{x},\gamma_m)=\tilde{w}_n(\bar{x},0)=0$ and $w_n\in C(E)$ we finally have that  $\lim_{m\to \infty}\tilde{w}_n(x_m,\gamma_m)=\tilde{w}_n(x,0)$. Since limit is the same for any other sequence $\gamma_m \to 0$ as $m\to \infty$ we have required continuity  and the proof is completed.
\end{proof}

\section{Stability with respect to controls}
In this section we do not assume controlled Feller property for the transition operator $\bP_n^a$.   Denote by ${\cal U}$ the class of Borel measurable functions $u:E \mapsto U$. For a sequence $u_n\in {\cal U}$ define define the operators $T_n^{u_n}$ and $\tilde{T}_n^{u_n}$ in the following way
 \begin{equation}\label{opTc}
T_n^{u_n}v(x) := c_n(x,u_n(x))+\int_E v(y)\bP_n^{u_n(x)}(x,dy)
\end{equation}
for $v\in B(E)$ and
\begin{equation}\label{Optrc}
\tilde{T}_n^{u_n} v(x,\gamma):= c_n(x,u_n(x))+{1\over \gamma}\ln\int_E e^{\gamma v(y,\gamma)}\bP_n^{u_n(x)}(x,dy)
\end{equation}
for $v(\cdot,\gamma)\in B(E)$ with fixed $\gamma\neq 0$. We have

\begin{proposition}\label{averpc}
Assume that $\Delta_n$ defined in \eqref{defD} satisfies (iii) of Theorem \ref{mtech} and $\sup_n R_n<\infty$. Then for any sequence $u_n\in {\cal U}$ there exist unique: functions  $w_n^{u_n}\in B(E)$ such that for a fixed $\bar{x}\in E$ we have  $w_n^{u_n}(\bar{x})=0$ and constants $\lambda_n(u_n)$ which are solutions to so called equation Poisson equation
\begin{equation}\label{Pois1}
w_n^{u_n}(x)= c_n(x,u_n(x))-\lambda_n(u_n)+\int_E w_{n+1}^{u_{n+1}}(y)\bP_n^{u_n(x)}(x,dy).
\end{equation}
Furthermore
\begin{equation}\label{impinec}
\|T_n^{u_n}T_{n+1}^{u_{n+1}}\ldots T_{n+k-1}^{u_{n+k-1}}0-w_n\|_{sp}\leq  \Delta_n\ldots \Delta_{n+k-1}\sup_i R_i.
\end{equation}
\end{proposition}
\begin{proof} We use Proposition \ref{mtechp} to the operator $T_n^{u_n}$ in a similar way as in
Theorem \ref{averp}.
\end{proof}

\begin{proposition}\label{averprc}
Assume that $\Delta_n^\gamma$ defined in Lemma \ref{defDrr} satisfies (iii) of Theorem \ref{mtech} and $\sup_n R_n<\infty$. Then for any sequence $u_n\in {\cal U}$ and fixed $\gamma\neq 0$  there exist unique: functions  $\tilde{w}_n^{u_n}(\cdot,\gamma)\in B(E)$ such that for a fixed $\bar{x}\in E$ we have  $\tilde{w}_n^{u_n}(\bar{x},\gamma)=0$ and constants $\lambda_n^\gamma(u_n)$ which are solutions to the equation
\begin{equation}\label{Pois2}
\tilde{w}_n^{u_n}(x,\gamma)=c_n(x,u_n(x))-\lambda_n^\gamma(u_n)+ {1\over \gamma}\ln\int_E e^{\gamma \tilde{w}_{n+1}^{u_{n+1}}(y,\gamma)}\bP_n^{u_n(x)}(x,dy).
\end{equation}
In addition
\begin{equation}\label{impinecc}
\|\tilde{T}_n^{u_n}\tilde{T}_{n+1}^{u_{n+1}}\ldots \tilde{T}_{n+k-1}^{u_{n+k-1}}0(\cdot,\gamma)-\tilde{w}_n(\cdot,\gamma)\|_{sp}\leq \Delta_n\ldots \Delta_{n+k-1}sup_i R_i.
\end{equation}
\end{proposition}
\begin{proof} We follow the arguments of the proof of Theorem \ref{averpr} with operator $\tilde{T}_n^{u_n}$ defined in \eqref{Optrc}.
\end{proof}

Consider pointwise convergence topology in ${\cal U}$ i.e. ${\cal U} \ni u^m \to u\in {\cal U}$ whenever $u^m(x)\to u(x)$ for $x\in E$ and $m\to \infty$. We extend this convergence to sequences of controls $u_n\in {\cal U}$ for $n\in \mathbb{N}_0$. We write $(u_n^m)\to (u_n)$ as $m\to \infty$ whenever
$u_n^m\to u_n$ for each $n\in  \mathbb{N}_0$.

We shall need the following continuity of the transition operators
\begin{equation}\label{varcont}
U\ni a \mapsto \bP_n^a(x,\cdot) \quad \mbox{is continuous in variation norm}
\end{equation}
for each $n\in \mathbb{N}$, which means that $a^m\to a$ implies that $\bP_n^{a^m}(x,\cdot)$ converges to $\bP_n^{a}(x,\cdot)$ in variation norm. We also assume that for $n\in \mathbb{N}_0$
\begin{equation}\label{ucontc}
\sup_x |c_n(x,a^m)-c_n(x,a)| \quad \mbox{$a^m \to a$ as $m\to \infty$}
\end{equation}

\begin{lemma}\label{impol}
Under \eqref{varcont}  when $(u_n^m)\to (u_n)$ as $m\to \infty$ we have for $n\in \mathbb{N}_0$ and $k\in \mathbb{N}$ and $x\in E$
\begin{equation}
T_n^{u_n^m}T_{n+1}^{u_{n+1}^m}\ldots T_{n+k-1}^{u_{n+k-1}^m}0(x) \to T_n^{u_n}T_{n+1}^{u_{n+1}}\ldots T_{n+k-1}^{u_{n+k-1}}0(x)
\end{equation}
Additionally under \eqref{ucontc} the above convergence is also in the span norm.
\end{lemma}
\begin{proof} The proof is by induction as in Proposition 2 of \cite{Ste2024} using Proposition 1 of \cite{Ste2024}. To have convergence in the span norm we need uniform in $x$ continuity of $c_n$ with respect to the second variable which is guaranteed by \eqref{ucontc}.
\end{proof}
The strategy $V=(u_0(X_0),u_1(X_1),\ldots,u_n(X_n),\ldots)$ in what follows will be denoted shortly by $V=((u_n))$.
\begin{theorem} Under \eqref{varcont} and \eqref{ucontc} assuming additionally that $\Delta_n$ defined in \eqref{defD} satisfies (iii) of Theorem \ref{mtech} and $\sup_n R_n<\infty$ when  $(u_n^m)\to (u_n)$ as $m\to \infty$ we have that
$w_n^{u_n^m}\to w_n^{u_n}$ and $\lambda_n(u_n^m)\to \lambda_n(u_n)$ for each $n\in \mathbb{N}$ as $m\to \infty$, where $(w_n^{u_n^m})$, $(\lambda_n(u_n^m))$ and $(w_n^{u_n})$, $(\lambda_n(u_n))$ are solutions to the Poisson equations \eqref{Pois1} corresponding to Markov controls $(u_n^m)$ and $(u_n)$ respectively. Consequently  $J_x((u_n^m))\to J_x((u_n))$ for $x\in E$ and $m\to \infty$.
\end{theorem}
 \begin{proof}
We have
\begin{eqnarray}\label{os1}
&&\|w_n^{u_n^m}-w_n^{u_n}\|_{sp}\leq \|w_n^{u_n^m}-T_n^{u_n^m}T_{n+1}^{u_{n+1}^m}\ldots T_{n+k-1}^{u_{n+k-1}^m}0\|_{sp} + \nonumber \\
&&\|T_n^{u_n^m}T_{n+1}^{u_{n+1}^m}\ldots T_{n+k-1}^{u_{n+k-1}^m}0(x)- T_n^{u_n}T_{n+1}^{u_{n+1}}\ldots T_{n+k-1}^{u_{n+k-1}}0(x)\|_{sp} \nonumber \\
&&+\|T_n^{u_n}T_{n+1}^{u_{n+1}}\ldots T_{n+k-1}^{u_{n+k-1}}0(x)-w_n^{u_n}\|_{sp}= a_n^m+b_n^m+d_n
\end{eqnarray}
Clearly using \eqref{impinec} we obtain that
\begin{equation}\label{os2}
0\leq a_n^m \vee d_n \leq  \Delta_n\ldots \Delta_{n+k-1}\sup_i R_i.
\end{equation}
By Lemma \ref{impol} we also have that $b_n^m\to 0$ as $m\to \infty$.
Since $\|w_n^m-w_n\|_{sp}\to 0$ as $m\to \infty$ we also have a convergence in supremum norm. Then by the Poisson equation \eqref{Pois1} we also have that $\lambda_n(u_n^m)\to \lambda_n(u_n)$
and $J_x((u_n^m))\to J_x((u_n))$, as $m\to \infty$, which completes the proof.
\end{proof}

\begin{theorem} Assume that $\Delta_n^\gamma$ defined in Lemma \ref{defDrr} satisfies (iii) of Theorem \ref{mtech} and $\sup_n R_n<\infty$. Then under \eqref{varcont} and \eqref{ucontc}
 when  $(u_n^m)\to (u_n)$ as $m\to \infty$ we have that
$\tilde{w}_n^{u_n^m}\to \tilde{w}_n$ and $\lambda_n^\gamma(u_n^m)\to \lambda_n^\gamma(u_n)$ for each $n\in \mathbb{N}$ and $\gamma \neq 0$ as $m\to \infty$, where $(w_n^{u_n^m})$, $(\lambda_n^\gamma(u_n^m))$ and $(\tilde{w}_n^{u_n})$, $(\lambda_n^\gamma(u_n))$ are solutions to the Poisson equations \eqref{Pois2} corresponding to Markov controls $(u_n^m)$ and $(u_n)$ respectively. Moreover $J_x^\gamma((u_n^m))\to J_x^\gamma((u_n))$ for $x\in E$ and $m\to \infty$.
\end{theorem}
\begin{proof} We have an analog of Lemma \ref{impol}, namely for $n\in \mathbb{N}_0$ and $k\in \mathbb{N}$ and $x\in E$
\begin{equation}
\tilde{T}_n^{u_n^m}\tilde{T}_{n+1}^{u_{n+1}^m}\ldots \tilde{T}_{n+k-1}^{u_{n+k-1}^m}0(x) \to \tilde{T}_n^{u_n}\tilde{T}_{n+1}^{u_{n+1}}\ldots \tilde{T}_{n+k-1}^{u_{n+k-1}}0(x)
\end{equation}
as $m\to \infty$. Then we follow \eqref{os1} and \eqref{os2} with operators $T_n^{u_n^m}$ and $T_n^{u_n}$ replaced by $\tilde{T}_n^{u_n^m}$ and $\tilde{T}_n^{u_n}$ respectively. The remaining part of the proof follows from the Poisson equation \eqref{Pois2}.
\end{proof}

\bibliographystyle{cas-model2-names}

\begin{thebibliography}{10}

\bibitem{Bao}
{\sc W. Bao, Y. Wang, J. Cheng, D. Zhang, W. Qi}, {\em Neutral network-based dynamic output feedback control for nonhomogeneous Markov swithching systems under deception attacks}, J. Franklin Inst. 362 (2025), 107502. 

\bibitem{Cao}
{\sc Xi-Ren Cao}, {\em Foundations of Average-Cost nonhomogeneous Controlled Markov Processes},
Springer 2021.

\bibitem{DiMasi 1999}
{\sc G.B. Di Masi, {\L}. Stettner}, {\em Risk-sensitive control of discrete-time {M}arkov processes with infinite horizon}, SIAM J. Control Optim.  38 (1999), pp. 61--78.
	
\bibitem{Ding} {\sc Y. Ding, H. Liu and K. Shi}, {\em $H_\infty$ stat-feedback controller design for continuous-time nonhomogeneous Markov jump systems}, Optim. Control Appl. Meth. 38 (2017), 133--144.

\bibitem{Dob}
{\sc R. L. Dobrusin}, {\em Central limit theorem for nonstationary Markov chains I, II.},
Theor. Probability Appl. 1 (1956), 65--80 and 329--383.

\bibitem{Dupuis}
{\sc P. Dupuis, R.S. Ellis}, {\em A Weak Convergence Approach to the Theory of Large Deviations}, Wiley, 1997.	

\bibitem{Ios}
{\sc M. Iosifescu}, {\em On two recent papers on ergodicity in nonhomogeneous Markov chains}, Annals Math. Stat. 43 (1972), pp. 1732--1736, 

\bibitem{LiuLu}
{\sc Z. Liu, D. Lu}, {\em Ergodicity of inhomogeneous Markov processes under general critiria},
Fron. Math. (2025), published online. 


\bibitem{Ma} {\sc  P. Massart}, {\em Concentration Inequalities and Model Selection}, Ecole d'Et\'e de Probabili\'es de Saint-Flour XXXIII - 2003, Springer 2007.

\bibitem{Park} {\sc Y. Park, J.C. Bean and R.L. Smith}, {\em Optimal Average Value Convergence in Nonhomogeneous Markov Decision Processes}, J Math Anal Appl 179 (1993), 525--536. 

\bibitem{Ngoc}
{\sc Ngoc Hoai An N. and Sung Hyun K.}, {\em Asynchronous $H_\infty$ observer-based control synthesis of nonhomogeneous Markovian jump systems with generalized incomplete transition rates}, Appl. Math. Comput. 411 (2021), 126532.

\bibitem{PitSte2022}
{\sc M. Pitera and {\L}. Stettner}, {\em Discrete-Time Risk Sensitive Portfolio Optimization with Proportional Transaction Costs}, Math. Finance, 33 (2023) no. 4, 1287--1313.


\bibitem{Ste2023}
{\sc {\L}. Stettner}, {\em Certainty equivalent control of discrete time Markov processes with the average reward functional}, Systems \& Control Letters 181 (2023), 105627.

\bibitem{Ste2024} {\sc {\L}. Stettner}, {\em Stability of long run functionals with respect to stationary Markov controls}, 2024 IEEE 63rd Conference on Decision and Control (CDC)
December 16-19, 2024. MiCo, Milan, Italy, 1832--1837, 979-8-3503-1632-2/24/\$31.00 2024 IEEE.

\bibitem{Zheng} {\sc Z. Zeng, X. Guo}, {\em Zero-Sum Non-stationary Stochastic Games with the Long-Run Average Criterion}, Appl. Math. Optim. (2024) 90:43.
\end{thebibliography}

\end{document}